\documentclass[12pt]{article}
\usepackage{latexsym, amsmath, amsfonts, amsthm, amssymb}
\usepackage{times}
\usepackage{a4wide}
\usepackage{url} 

\def \tr {\text{\rm Tr}}

\def \Ent {\text{\rm Ent}}
\def \D {\, || \,}

\def \RR {\mathbb R}

\def \EE {\mathbb E}

\def \eps {\varepsilon}

\def \cE {\mathcal E}

\def \Tr {{ \rm Tr }}

\newtheorem{theorem}{Theorem}[section]

\newtheorem{lemma}[theorem]{Lemma}

\newtheorem{proposition}[theorem]{Proposition}
\newtheorem{corollary}[theorem]{Corollary}

\theoremstyle{definition}

\def\myffrac#1#2 in #3{\raise 2.6pt\hbox{$#3 #1$}\mkern-1.5mu\raise 0.8pt\hbox{$
		#3/$}\mkern-1.1mu\lower 1.5pt\hbox{$#3 #2$}}

\def\qed{\hfill $\vcenter{\hrule height .3mm
		\hbox {\vrule width .3mm height 2.1mm \kern 2mm \vrule width .3mm
			height 2.1mm} \hrule height .3mm}$ \bigskip}

\def \id {{\rm Id}}

\def \cov {{ \rm Cov}}

\def \deltaEPI {\delta_{\rm KL}}

\begin{document}

\title{Affirmative Resolution of Bourgain's Slicing Problem \\ using Guan's Bound}
\author{Boaz Klartag and Joseph Lehec}
\date{}
\maketitle

\abstract{We provide the final step in the 
resolution of Bourgain's slicing problem in the affirmative. Thus we establish the following theorem: 
for any convex body $K \subseteq \RR^n$ of volume one, there exists a hyperplane $H \subseteq \RR^n$
such that $$ Vol_{n-1}(K \cap H) > c,  $$ where $c > 0$ is a universal constant.
Our proof combines Milman's theory of $M$-ellipsoids, stochastic localization with a recent bound by Guan, 
and stability estimates for the Shannon-Stam  inequality by Eldan and Mikulincer. 
}

\section{Introduction}

Recently, a transformative paper by Qingyang Guan was posted on arXiv, providing a solution to Bourgain's slicing problem
 up to a factor of $\log \log n$. While Bourgain \cite{Bou1, Bou2} poses the question of whether $L_n \leq C$ for a universal constant $C > 0$
(see (\ref{eq_847}) below for the definition of $L_n$), it is shown in Guan \cite{guan} that 
$$ L_n \leq C \log \log n. $$
This improves upon the estimate $L_n \leq C \sqrt{\log n}$ proved in \cite{root_log}. For many years,
the best-known bounds were $L_n \leq C n^{1/4} \log n$ established in Bourgain \cite{B1, B2} and $L_n \leq C n^{1/4}$ from \cite{K_quarter}.
A breakthrough was achieved four years ago by Yuansi Chen \cite{chen}, 
who showed that $L_n \leq \exp( C \sqrt{ \log n \cdot \log \log n} )$. Chen's result was subsequently 
improved to $L_n \leq C \log^4 n$ in \cite{KL0}, then to $L_n \leq C \log^{2.223...} n$
in Jambulapati, Lee and Vempala \cite{JLV} and then to $L_n \leq C \log^{2.082...}$ and $L_n \leq C \sqrt{\log n}$, see \cite{root_log}.

\medskip 
As with all recent advances toward the slicing problem, Guan's work builds upon 
a technique named {\it stochastic localization} that was invented 
by Ronen Eldan in his Ph.D. thesis \cite{eldan, eldan_phd} and refined by Lee and Vempala \cite{LV} and others. In our context, this technique 
involves applying stochastic analysis in order to provide precise estimates 
related to the heat evolution of a probability measure in $\RR^n$, under convexity assumptions. 
A central estimate in Guan's paper is the bound in \cite[Lemma 2.1]{guan},
\begin{equation}  \EE \tr[A_t^2] \leq C n \qquad \qquad \qquad (t > 0), 
\label{eq_1749} \end{equation}
where $C > 0$ is a universal constant. See Section \ref{sec_1116} below for an explanation of this notation.
While Chen's work \cite{chen} relies on  manipulation of $3$-tensors using the log-concave Lichnerowicz inequality, Guan's proof of (\ref{eq_1749}) 
employs the  improved Lichnerowicz inequality from \cite{root_log} for analyzing these tensors. As it turns out,  the bound  (\ref{eq_1749}) provides the missing link 
in an approach to Bourgain's slicing problem 
discussed by the authors a few years ago.
We thus complete the proof of the following:

\begin{theorem} For any convex body $K \subseteq \RR^n$ of volume one, there exists a hyperplane $H \subseteq \RR^n$
such that $Vol_{n-1}(K \cap H) > c$. Here $c > 0$ is a universal constant.
\label{thm_1744_}
\end{theorem}

Bourgain's slicing problem, also known as the hyperplane conjecture, admits several equivalent formulations; see e.g.
Brazitikos, Giannopoulos, Valettas  and  Vritsiou \cite{BGVV}, Klartag and Milman \cite{KM_B} or Milman and Pajor \cite{MP} for background on the slicing problem.
One such formulation focuses on the relationship between two different measures of the ``size'' of a convex body: the volume of the convex body and the determinant of its covariance matrix. For a probability measure $\mu$ on $\RR^n$ with finite second moments we write
$\cov(\mu) = (\cov_{ij}(\mu))_{i,j=1,\ldots,n} \in \RR^{n \times n}$ for its covariance matrix, given by 
$$\cov_{ij}(\mu) = \int_{\RR^n} x_i x_j d \mu(x) - \int_{\RR^n} x_i d \mu(x) \int_{\RR^n} x_j d \mu(x). $$
The covariance matrix is a positive semi-definite, symmetric matrix. 
For a convex body $K \subseteq \RR^n$ (i.e., a compact, convex set with a non-empty interior)  we write $\lambda_K$ for the uniform 
probability measure on~$K$. Abbreviate $\cov(K) =\cov(\lambda_K)$. The {\it isotropic constant} of the convex body $K \subseteq \RR^n$ is defined to be
	$$ L_K := \left( \frac{\det\cov(K)}{Vol_n(K)^2} \right)^{\frac{1}{2n}}. $$
The isotropic constant is an affine invariant, that is, the isotropic constant of $K$ equals that of $T(K)$ for any affine, invertible map $T: \RR^n \rightarrow \RR^n$.
 Define 
\begin{equation}
L_n = \sup_{K \subseteq \RR^n} L_K, \label{eq_847} 
\end{equation}
where the supremum runs over all convex bodies $K \subseteq \RR^n$. 
In one of its formulations (see e.g. \cite{KM_B}), Bourgain's slicing problem asks whether $L_n < C$ for a universal constant $C > 0$.  
Consequently, Theorem \ref{thm_1744_} follows from the following:
\begin{theorem} $\displaystyle \sup_{n \geq 1} L_n < \infty$. 
\label{thm_1744}
\end{theorem}

In addition to stochastic localization and Guan's bound, our proof of Theorem \ref{thm_1744} uses 
the stability estimates for the Shannon-Stam  inequality  by Eldan and Mikulincer \cite{EM},
which, in turn, are based on the stochastic proof of this inequality from \cite{lehec}.
In fact, it was suggested already in Ball and Nguyen \cite{BN} that both the deficit in the Shannon-Stam  inequality 
and the evolution under the heat flow are relevant to the slicing problem. The
culminating new ingredient in our argument, however, is the use of Milman's theory of $M$-ellipsoids.
Indeed, Bourgain's slicing problem can be viewed as a strengthening of Milman's theory,
and tools from information theory, heat flow and stochastic analysis allow us to leverage this theory 
and prove Theorem \ref{thm_1744}. 

\medskip In principle, one could extract from our proof of Theorem \ref{thm_1744} an explicit -- though absurdly large -- upper bound for $\sup_n L_n$.
A strong version of the slicing problem asks whether the supremum in (\ref{eq_847}) is attained when $K \subseteq \RR^n$ is a simplex.
If the answer is affirmative and indeed $\sup_n L_n = 1/e$, this would imply Mahler's conjecture on the product of the volume of a convex 
body and the volume of its  polar body, see \cite{K_advances}. There is also a strong version of the slicing problem for centrally-symmetric convex 
bodies, which asks whether the supremum in (\ref{eq_847}), when restricted to centrally-symmetric convex bodies (i.e., $K = -K$), is attained for the cube.
If true, this would imply the Minkowski lattice conjecture, see Magazinov \cite{mag}.

\medskip 
A function $f: \RR^n \rightarrow [0, \infty)$ is log-concave if $K = \{ x \in \RR^n \, ; \, f(x) > 0 \}$
is a convex set, and additionally the function $-\log f: K \rightarrow \RR$ is convex.
A log-concave probability density in $\RR^n$ necessarily decays exponentially at infinity, and consequently 
admits moments of all orders (see e.g. \cite[Lemma 2.2.1]{BGVV}).
A Borel probability measure in $\RR^n$ is log-concave if there exists an affine subspace $E \subseteq \RR^n$
such that $\mu$ is supported in $E$, and is absolutely-continuous with a log-concave 
density relative to this affine subspace $E$.  The uniform probability measure on any convex body in $\RR^n$ is log-concave, as well as all Gaussian measures.
Suppose that $\mu$ is an absolutely-continuous probability measure in $\RR^n$ with density $f$. The differential entropy 
of $\mu$ is
$$ \Ent(\mu) = -\int_{\RR^n} f \log f. $$
For an absolutely-continuous, log-concave probability measure $\mu$ in $\RR^n$ we define its isotropic constant via
\begin{equation}
 L_{\mu} := e^{-\Ent(\mu) / n} \cdot \det \cov(\mu)^{1/(2n)}.
\label{eq_1750} \end{equation}
Note that $L_K = L_{\lambda_K}$. When $\mu$ is a log-concave probability measure in $\RR^n$ which is not-necessarily absolutely-continuous, 
we consider the affine subspace $E \subseteq \RR^n$ in which $\mu$ is absolutely-continuous, and 
define $L_{\mu}$ relative to the affine subspace $E$. It follows that for any one-to-one affine map $T: \RR^n \rightarrow \RR^N$ and any log-concave probability measure $\mu$ in $\RR^n$,
$$ L_{\mu} = L_{T_* \mu}, $$
where $T_* \mu$ is the push-forward of $\mu$ under the map $T$. 
When the absolutely-continuous, log-concave probability measure $\mu$ in $\RR^n$ is centered, i.e., 
when $\int x d \mu(x) = 0$, we have
\begin{equation}  -\log f(0) \leq \Ent(\mu) \leq -\log f(0) + n,
\label{eq_1048} \end{equation}
where $f$ is the log-concave density of $\mu$. See e.g. \cite[Lemma 83]{KL} for a proof of (\ref{eq_1048}).
It follows from (\ref{eq_1750}) and (\ref{eq_1048}) that for a centered probability measure $\mu$ in $\RR^n$ with a log-concave density $f$,
\begin{equation}
 L_{\mu} \leq f(0)^{1/n} \cdot \det \cov(\mu)^{1/(2n)} \leq e \cdot L_{\mu}.
\label{eq_912} \end{equation}
By combining Theorem \ref{thm_1744} with a result by Ball \cite{Ball_studia} (see also \cite{K_quarter} for the non-even case), we conclude that  
for any log-concave probability measure $\mu$ in any finite-dimensional linear space,
\begin{equation}  \frac{1}{\sqrt{2 \pi e}} \leq L_{\mu} \leq C, 
\label{eq_1910} \end{equation}
where $C > 0$ is a universal constant. See e.g. \cite[Section 9]{KL}
for the inequality on the left-hand side of (\ref{eq_1910}), in which equality is attained 
when $\mu$ is a Gaussian measure. The mathematical literature contains two slightly different 
notions  of an ``isotropic log-concave measure'':
\begin{enumerate}
\item In Bourgain's normalization, 
one says that a convex body $K \subseteq \RR^n$ is a {\it convex isotropic body of volume one} if $\lambda_K$ is centered, 
$Vol_n(K) = 1$ and $\cov(K)$ is a scalar matrix. In this case, we have $\cov(K) = L_K^2 \cdot \id$.

\item In the probabilistic normalization (going back at least to Kannan, Lov\'asz and Simonovitz \cite{KLS}), one says that a log-concave probability measure
$\mu$ in $\RR^n$ is {\it isotropic with identity covariance} if $\mu$ is centered and $\cov(\mu) = \id$.
\end{enumerate}
In this paper, unless stated otherwise, the term {\it isotropic} refers to {\it isotropic with identity covariance}. 

\medskip Throughout this paper, we write $c, C, c', \tilde{C}, c_1, C_2$ etc. for various positive universal constants 
whose value may change from one line to the next.
We write $X \lesssim Y$ for two expressions $X$ and $Y$ if 
$c X \leq Y$, where $c > 0$ is a universal constant. If $X \lesssim Y$ and $Y \lesssim X$ then we write $$ X \sim Y. $$
The Euclidean norm of $x = (x_1,\ldots,x_n) \in \RR^n$ is denoted by $|x| = \sqrt{ \sum_i x_i^2}$, and $x \cdot y = \sum_i x_i y_i$ for $x,y \in \RR^n$. We write $Vol_n$ for $n$-dimensional volume, 
$B^n = \{ x \in \RR^n \, ; \, |x| \leq 1 \}$ is the unit ball centered at the origin, and $\log$ is the natural logarithm. The orthogonal
projection operator onto a subspace $E \subseteq \RR^n$ is denoted by $Proj_E: \RR^n \rightarrow E$.
Our notation does not distinguish between a linear operator $A: \RR^n \rightarrow \RR^n$ and the matrix $[A] \in \RR^{n \times n}$  
satisfying $A (v) = [A]  v$ for any $v \in \RR^n$. Consequently we may occasionally write that a probability measure $\mu$ in $\RR^n$ satisfies 
$$ \cov(\mu) = Proj_E $$
for a subspace $E \subseteq \RR^n$. When we write $A \leq B$ for two symmetric matrices $A,B \in \RR^{n \times n}$
we mean, of course, that $A x \cdot x \leq B x \cdot x$ for all $x \in \RR^n$. For $x \in \RR^n$
we write $x \otimes x = (x_i x_j)_{i,j=1,\ldots,n} \in \RR^{n \times n}$, and $\tr[A]$ is the trace of the matrix $A$
while $|A| = \sqrt{\tr[A^* A]}$ for $A^*$ being the transpose of $A$. We write $\nabla^2 f(x) \in \RR^{n \times n}$ for the Hessian
of the smooth function $f: \RR^n \rightarrow \RR$ at the point $x \in \RR^n$.

\medskip The rest of this paper is organized as follows: In Section \ref{sec_1116} we introduce 
stochastic localization and apply Milman's theory of $M$-ellipsoids as well as Guan's bound. 
Section \ref{sec_1547} is concerned with the Shannon-Stam inequality and its stability estimates by Eldan and Mikulincer \cite{EM}.
Theorem \ref{thm_1744} is finally proved in Section \ref{sec_thmproof}. 

\medskip
\emph{Acknowledgements}. BK would like to thank Pierre Bizeul, Ronen Eldan and Vitali Milman, as well as the late Jean Bourgain, for illuminating 
discussions on the slicing problem over the years, and was supported by a grant from the Israel Science Foundation (ISF).

\section{Heat Flow and Stochastic Localization}
\label{sec_1116}

Suppose that $\mu$ is an isotropic, log-concave probability measure in $\RR^n$. Write $\rho$ for the log-concave density of $\mu$
and denote $$  p_{t, \theta}(x) = \frac{1}{Z(t,\theta)} e^{\theta \cdot x  - t |x|^2/2} \rho(x) \qquad \qquad (t \geq 0, \theta \in \RR^n,x \in \RR^n),
$$
where the normalizing constant $Z(t, \theta) > 0$ ensures that $p_{t, \theta}$ is a probability density in $\RR^n$.
The barycenter of $p_{t, \theta}$ is the vector $a(t, \theta) \in \RR^n$ defined via
$$ a(t,\theta) =  \int_{\RR^n} x p_{t, \theta}(x) \, dx. $$
The covariance matrix $A(t,\theta) \in \RR^{n \times n}$ is of course
$$ A(t,\theta) = \int_{\RR^n} (x \otimes x) p_{t, \theta}(x) \, dx \, - \, a(t, \theta) \otimes a(t, \theta). $$
Consider the stochastic process $(\theta_t)_{t \geq 0}$ with  initial condition $\theta_0 = 0$ that satisfies the stochastic differential equation
\begin{equation}  d \theta_t = d W_t + a(t, \theta_t) dt \qquad \qquad \qquad (t > 0). \label{eq_1012} \end{equation}
Here, $(W_t)_{t \geq 0}$ is a standard Brownian motion in $\RR^n$ with $W_0 = 0$. 
The strong solution to this stochastic differential equation
exists and is unique (see, e.g. Chen \cite{chen} or \cite[Section 6]{KL}). The stochastic process $(\theta_t)_{t \geq 0}$ is used  in the theory of non-linear filtering, 
particularly for estimating a constant unknown signal, see e.g. Chiganski
\cite[Chapter~6]{chiganski}. For basic properties of stochastic localization, see the lecture notes \cite{KL} or the papers by Chen \cite{chen}, Eldan \cite{eldan} and Lee and Vempala \cite{LV}, as well as \cite{root_log, KL0, KP}. 

\medskip As explained e.g. in Klartag and Putterman \cite{KP}, the process $(\theta_t)_{t \geq 0}$ has the same distribution
as the process $(t X + W_t)_{t \geq 0}$
where $X$ is a random vector with law $\mu$, independent of the Brownian motion $(W_t)_{t \geq 0}$.
Moreover, for any $t \geq 0$, the probability density $p_{t, t X + W_t}$ is the density of the conditional law of $X$ given $t X + W_t$.
This fact is explained e.g. in \cite[Sections 5 and 6]{KL} along with the following corollary:
\begin{equation}  \EE  p_{t, \theta_t}(x) = \rho(x) \qquad \qquad \qquad (t \geq 0, x \in \RR^n). \label{eq_2326} \end{equation}
We abbreviate 
$$ a_t = a(t, \theta_t) \qquad \text{and} \qquad A_t = A(t, \theta_t) \geq 0. $$
We refer to $(A_t)_{t \geq 0}$ as the {\it covariance process} associated with stochastic localization starting from the probability measure $\mu$.
From (\ref{eq_2326}) we obtain the  decomposition of variance:
\begin{equation}  \EE A_t + \EE a_t \otimes a_t = A_0 = \cov(\mu) = \id, \label{eq_2309} \end{equation}
as $\mu$ is isotropic. We will  use the following differentiation formula (e.g. \cite[Lemma 53]{KL}):
\begin{equation}
\frac{d}{dt} \EE A_t = - \EE A_t^2.
\label{eq_1645}
\end{equation}
The log-concavity of $\rho$ and the  Lichnerowicz inequality imply that almost surely, for $t > 0$,
\begin{equation}
A_t \leq \frac{1}{t} \cdot \id,
\label{eq_1711}
\end{equation}
because $A(t, \theta) \leq (1/t) \cdot \id$ for any $t > 0$ and $\theta \in \RR^n$. Inequality (\ref{eq_1711}) and its improvements are discussed in \cite{root_log}. 
For $s > 0$, we write $\gamma_s$ for the Gaussian probability measure in $\RR^n$ of mean zero and covariance $s \cdot \id$.

\begin{lemma} For any $t > 0$,
$$ \EE \left(a_t - \frac{\theta_t}{1+t} \right) \otimes \theta_t = 0. $$
\label{lem_1213}
\end{lemma}

\begin{proof} For $s > 0$ we write $\mu_s = \mu * \gamma_s$ for the convolution of $\mu$ 
and $\gamma_s$. Let $e^{-\psi_s}$ be the positive, 
smooth density of $\mu_s$ in $\RR^n$. Integrating by parts yields
\begin{equation}  \int_{\RR^n} [\nabla \psi_s(y) \otimes y ] d \mu_s(y) = -\int_{\RR^n} \left[ \left( \nabla e^{-\psi_s(y)} \right) \otimes y \right] dy = \id, \label{eq_1714} \end{equation}
where no boundary terms arise as $e^{-\psi_s}$ is a smooth, positive function in $\RR^n$ decaying exponentially at infinity.
The first displayed formula in the proof of Lemma 2.2 in \cite{KP} states that
for any $t > 0$ and $\theta \in \RR^n$, with $s = 1/t$ and $y = \theta / t \in \RR^n$,
$$  \nabla \psi_s(y) = \frac{y - a(t, \theta)}{s}. $$
Hence,
\begin{equation}
 a(t, \theta) - \frac{\theta}{1+t} = \frac{s y}{s+1}  - s \nabla \psi_s(y).
\label{eq_1446} \end{equation}
The random vector $\theta_t$ coincides in law with the random vector $t X + W_t$, where $X$ and $(W_t)_{t \geq 0}$ are as above.
Hence the law of the random vector $\theta_t / t$ 
is the convolution of $\mu$ with a Gaussian measure of mean zero and covariance $(1/t) \cdot \id = s \cdot \id$. 
In other words, the random vector $\theta_t / t$ has law $\mu_s$.  
Therefore, by (\ref{eq_1714}) and (\ref{eq_1446}), and since $\mu$ is centered,
\begin{align*}  \EE \left(a(t, \theta_t) - \frac{\theta_t}{1+t} \right) \otimes \frac{\theta_t}{t} & = \int_{\RR^n} \left[  \left( \frac{s y}{s+1} - s \nabla \psi_s(y)  \right) \otimes y \right] d \mu_s(y)  \\ & = \frac{s}{s+1} \cov(\mu_s) - s \id.
\end{align*}
However, $\cov(\mu_s) = (s+1) \cdot \id$ since $\mu$ is isotropic, and the lemma follows.
\end{proof}

For any subspace $E \subseteq \RR^n$, the probability measure
$$ \mu_E = (Proj_E)_* \mu $$
is log-concave by the Pr\'ekopa-Leindler inequality. Since $\cov(\mu) = \id$, the measure $\mu_E$
is isotropic relative to the subspace $E$, i.e., it is centered with 
\begin{equation} \cov(\mu_E) = Proj_E. \label{eq_1146} \end{equation}

\begin{lemma} Let $E \subseteq \RR^n$ be a subspace. 
Let $(A_{E, t})_{t \geq 0}$ be the covariance process
associated with  stochastic localization starting from the measure $\mu_E = (Proj_E)_* \mu$. Then for $t > 0$,
\begin{equation}  \EE A_{E,t} \geq Proj_E \cdot \EE A_t \cdot Proj_E. 
\label{eq_705} \end{equation}
\label{lem_1816}
\end{lemma}

\begin{proof}
Set $s = 1/t$. Let $X$ be a random vector with law $\mu$, 
and let $Z$ be a standard Gaussian random vector in $\RR^n$, independent of $X$.  
For $\theta \in \RR^n$ write $\mu_{t, \theta}$ for the probability measure in $\RR^n$ with density $p_{t, \theta}$.
We claim that 
the covariance matrix $A_t = \cov(\mu_{t, \theta_t})$  coincides in law with 
$$
\cov ( X \mid X+\sqrt s Z ), 
$$ 
the conditional covariance of $X$ given $X + \sqrt{s} Z$.
Indeed, this follows from the fact that the measure $\mu_{t, \theta_t}$ has the same distribution as the conditional law 
of $X$ given $X+\sqrt s Z$ (see~\cite[Sections 5 and 6]{KL}).
Consequently,
$$ 
\EE A_t = \EE \cov ( X \mid X+\sqrt s Z ).
$$
In a similar way, since $Proj_E Z$ is a standard Gaussian vector in $E$ 
independent of $Proj_E X$, 
\[ 
\begin{split}
\EE A_{E,t} & = \EE \cov ( Proj_E X \mid Proj_E ( X+\sqrt s Z) )  \\
& = Proj_E \cdot \EE \cov (  X \mid Proj_E ( X+\sqrt s Z) ) \cdot Proj_E .
\end{split} 
\] 
Since $Proj_E ( X+\sqrt s Z)$ is a function of $X+ \sqrt{s} Z$, it generates a smaller 
$\sigma$-field. Hence, 
 \begin{equation} 
\EE \cov ( X \mid X+\sqrt s Z )
\leq 
\EE \cov ( X \mid Proj_E ( X+\sqrt s Z) ).
\label{eq_906} \end{equation}
Multiplying (\ref{eq_906}) on both sides by $Proj_E$ yields the result. 
\end{proof}

Guan's bound (\ref{eq_1749}), established in \cite{guan}, is described in the following lemma\footnote{Notes on Guan's proof of this lemma are available at \url{https://tinyurl.com/3fwp3u22}}:

\begin{lemma}[Guan] Suppose that $\mu$ is an isotropic, log-concave probability measure in $\RR^n$. 
Let $(A_t)_{t \geq 0}$ be the covariance process associated with stochastic localization starting from $\mu$. Then
for any $t > 0$,
$$ \EE \tr[A_t^2] \leq C n,
$$
where $C > 0$ is a universal constant.
\label{guan}
\end{lemma} 

The remainder of this section 
is devoted to the formulation and proof of the following:

\begin{proposition} Let $n \geq 10$. Then there exists an integer $n/4 \leq m \leq n$, a real number $t_0 \sim 1$ and an isotropic, log-concave probability measure $\nu$ in $\RR^m$
such that $$ L_{\nu} \gtrsim L_n $$ and
\begin{equation}  \EE A_{t_0} \geq \frac{1}{4} \cdot \id. \label{eq_1148} \end{equation}
Here, $(A_t)_{t \geq 0}$ is the covariance process associated with stochastic localization starting from~$\nu$.
\label{prop_1115}
\end{proposition}

The proof of Proposition \ref{prop_1115} requires some preparation.
It was shown by Bourgain, Klartag and Milman \cite[Proposition 1.3]{BKM}
that if $m \leq n$ then 
$$
L_m \leq C \cdot L_n. 
$$
Thus in the proof of Proposition \ref{prop_1115},
we may slightly increase $n$ if needed, assume that $n \geq 12$ is a number divisible by six,
and prove the proposition with $m = n/3$.

\medskip 
Let $K \subseteq \RR^n$ be a convex isotropic body of volume one with
\begin{equation}
L_K = L_n,
\label{eq_848} \end{equation}
which exists since the supremum in (\ref{eq_847}) is attained. To avoid confusion, we emphasize that 
\begin{equation}
Vol_n(K) = 1,
\label{eq_1417}
\end{equation}
and $$ \cov(K) = L_K^2 \cdot \id. $$
Write $\mu$ for the uniform probability measure on the convex body $ L_K^{-1} \cdot K = \{ x / L_K \, ; \, x \in K \}$. Then $\mu$ is a centered probability measure with identity covariance.
Thus $\mu$ is an isotropic, log-concave probability measure.  
For two sets $A, B \subseteq \RR^n$ we write
$$ N(A,B) = \min \{ L \geq 1 \, ; \, \exists x_1,\ldots,x_L \in \RR^n \textrm{ such that } \ A \subseteq \cup_{i = 1}^L (x_i + B) \} $$
for the covering number of $A$ by $B$. Here, $x + B = \{ x + y \, ; \, y \in B \}$.

\begin{lemma} For any subspace $E \subseteq \RR^n$ of dimension $\dim(E) = n/3$, the measure 
$\mu_E = (Proj_E)_* \mu$ satisfies
$$ L_{\mu_E} \geq c \cdot L_n, $$
where $c > 0$ is a universal constant. \label{lem_1148}
\end{lemma}

\begin{proof} 
Note that the density of $\mu$ equals 
$L_K^n$ in the convex body $L_K^{-1} \cdot K$ in which it is supported. Write $\rho_E: E \rightarrow [0, \infty)$ for the log-concave density of $\mu_E$. Set $\ell = \dim(E) = n/3$. Then,
$$ \rho_E(0) = L_K^n \cdot Vol_{n-\ell}(L_K^{-1} \cdot K \cap E^{\perp}) = L_k^{\ell} \cdot Vol_{n - \ell}(K \cap E^{\perp}), $$
where $E^{\perp} \subseteq \RR^n$ is the orthogonal complement to $E$. According to (\ref{eq_912}), 
\begin{equation} L_{\mu_E} \sim \rho_E(0)^{1/\ell} \cdot (\det \nolimits_E \cov(\mu_E))^{1/\ell} = \rho_E(0)^{1/\ell} 
= L_K \cdot Vol_{n-\ell}(K \cap E^{\perp})^{1/\ell},
\label{eq_947} \end{equation}
where $\det_E(T)$ denotes the determinant of a linear operator $T: E \rightarrow E$, thus $\det_E \cov(\mu_E) = \det_E Proj_E = 1$
by (\ref{eq_1146}).

\medskip 
We will use the theory of $M$-ellipsoids from Milman \cite{M1}, see also  Milman and Pajor \cite{MP2}, Pisier \cite[Chapter 7]{pisier_book} 
or Brazitikos, Giannopoulos, Valettas  and  Vritsiou \cite[Chapter~1]{BGVV}. This theory implies that for any convex body $K \subseteq \RR^n$,
there exists an ellipsoid $\cE \subseteq \RR^n$ with $Vol_n(\cE) = Vol_n(K)$ such that
\begin{equation} \max \left \{ N(K, \cE), N(\cE, K) \right \} \leq e^{C n}. 
\label{eq_930} \end{equation}
As in Bourgain, Klartag and Milman \cite{BKM}, we write $\sqrt{n} \lambda_1 \leq \sqrt{n} \lambda_2 \leq \ldots \leq \sqrt{n} \lambda_n$
for the lengths of the axes of the ellipsoid $\cE \subseteq \RR^n$. Then
\begin{equation}  \left( \prod_{i=1}^n \lambda_i \right)^{1/n} = Vol_n(\cE)^{1/n} \cdot 2Vol_n(B^n)^{-1/n} / \sqrt{n} =  2Vol_n(B^n)^{-1/n} / \sqrt{n}\sim 1, 
\label{eq_1439} \end{equation}
where we used the fact that $Vol_n(\cE) = Vol_n(K) = 1$, thanks to (\ref{eq_1417}).  According 
to  \cite[Corollary 3.5]{BKM}, the maximality property (\ref{eq_848}) of the isotropic constant implies 
that for any $1 \leq k \leq n$ and any $k$-dimensional subspace $E \subseteq \RR^n$, 
\begin{equation} Vol_{k}(K \cap E)^{1/(n-k)} \leq C. \label{eq_932} \end{equation}
In \cite[Section 4.1]{BKM} it is shown (see the formula displayed before (4) in \cite{BKM}) that  (\ref{eq_930}), (\ref{eq_1439}) and  (\ref{eq_932}) imply that
\begin{equation}  \left( \prod_{i=1}^{n-\ell} \lambda_i \right)^{n - \ell}  = \left( \prod_{i=1}^{2n/3} \lambda_i \right)^{3/(2n)} \geq \left( \prod_{i=1}^{n/2} \lambda_i \right)^{2/n} > c. 
\label{eq_2138} \end{equation}
Among all $(n-\ell)$-dimensional central sections of a given ellipsoid, the central section of minimal volume is the one spanned by the shortest axes. 
This standard fact may be proved by the min-max characterization of the eigenvalues of a symmetric matrix.
Thus, for any subspace $F \subseteq \RR^n$ with $\dim(F) = n-\ell$, by (\ref{eq_2138}),
\begin{equation}  Vol_{n-\ell} ( \cE \cap F ) \geq Vol_{n-\ell} ( \cE \cap F_0 ) = Vol_{n-\ell}(B^{n-\ell}) \prod_{i=1}^{n-\ell} \left( \frac{\sqrt{n}}{2} \lambda_i \right) \geq c^n 
\label{eq_2139} \end{equation}
where $F_0 \subseteq \RR^n$ is the $(n-\ell)$-dimensional subspace spanned by the axes of the ellipsoid $\cE$ corresponding to $\lambda_1,\ldots,\lambda_{n-\ell}$.
By using (\ref{eq_930}) and (\ref{eq_2139}), we see that for any $(n-\ell)$-dimensional subspace $F \subseteq \RR^n$,
\begin{align} \nonumber \tilde{c}^n & \leq Vol_{n-\ell} ( \cE \cap F ) \leq N(\cE, K) \cdot \max_{x \in \RR^n} Vol_{n-\ell}( (x+K) \cap F )
\\ & \leq e^{C n} \cdot \max_{x \in \RR^n} Vol_{n-\ell}( (x+K) \cap F ). \label{eq_2203} \end{align}
Since $K$ is centered, $\max_{x \in \RR^n} Vol_{n-\ell}( (x+K) \cap F ) \leq C^n Vol_{n-\ell}(K \cap F)$ according to Fradelizi \cite{F}.
It thus follows from (\ref{eq_2203}) that for any $(n-\ell)$-dimensional subspace $F \subseteq \RR^n$,
\begin{equation}  
Vol_{n-\ell}( K \cap F ) \geq  c^n = \left( c^3 \right)^{\ell} = \tilde{c}^{\ell}. \label{eq_954} \end{equation}
In particular, (\ref{eq_954}) applies for the subspace $F = E^{\perp}$. From (\ref{eq_947}) and (\ref{eq_954}),
$$ L_{\mu_E} \gtrsim L_K = L_n $$
where the last passage is the content of (\ref{eq_848}).
\end{proof}

Recall the covariance process $(A_t)_{t \geq 0}$ that is associated with stochastic localization starting from the measure $\mu$. 
By (\ref{eq_1645}) and  Lemma \ref{guan}, 
\begin{equation}  \frac{d}{dt} \EE \Tr[A_t] = -\EE \Tr[A_t^2]\geq -C n \qquad \qquad (t > 0). 
 \label{eq_2205} \end{equation}
Set $c_0 = \min \{1/(2C), 1 \}$ where $C > 0$ is the universal constant from (\ref{eq_2205}). 
Since $A_0 = \id$ as $\mu$ is isotropic, 
\begin{equation}  \EE \Tr[A_t] = n - \int_0^t \EE \Tr[A_s^2] ds \geq n - \frac{1}{2C} \cdot C n 
= n /2 \qquad \qquad \qquad \text{for} \ t = c_0.
\label{eq_2310} \end{equation}

\begin{lemma} There exists a subspace $E \subseteq \RR^n$ with $\dim(E) = n/3$ such that for $t = c_0$,
\begin{equation} Proj_E \cdot \EE A_t \cdot Proj_E \geq \frac{1}{4} \cdot Proj_E, \label{eq_2322} \end{equation}
where $c_0 > 0$ is the universal constant from (\ref{eq_2310}). 
\label{lem_2326}
\end{lemma}

\begin{proof} From (\ref{eq_2309}) we know that 
 $$ \EE A_t \leq \id. $$ Write $0 < \lambda_1 \leq \lambda_2 \leq \ldots \leq \lambda_n$ for 
the eigenvalues of $\EE A_t$, repeated according to their multiplicity. Then $\lambda_i \leq 1$ for all $i$, while (\ref{eq_2310}) yields
\begin{equation} 
 \sum_{i=1}^n \lambda_i \geq n/2. 
 \label{eq_2319} \end{equation}
We claim that 
\begin{equation}  \lambda_{2n / 3} \geq 1/4. \label{eq_2320} \end{equation}
Indeed, otherwise $\lambda_i < 1/4$ for all $i \leq 2n/3$ and consequently
$$ \sum_{i=1}^n \lambda_i < \frac{1}{4} \cdot \frac{2n}{3} + 1 \cdot \frac{n}{3}  = \frac{n}{2}, $$
in contradiction to (\ref{eq_2319}). Hence (\ref{eq_2320}) is proved. Let $E \subseteq \RR^n$ be the $(n/3)$-dimensional subspace spanned 
by eigenvectors of $\EE A_t$ corresponding to the eigenvalues $\lambda_{2n / 3+1},\ldots,\lambda_n$. 
All of these 
eigenvalues are not smaller than $1/4$, according to (\ref{eq_2320}). This implies (\ref{eq_2322}).
\end{proof}

\begin{proof}[Proof of Proposition \ref{prop_1115}] Let $E \subseteq \RR^n$ be the $(n/3)$-dimensional subspace whose 
existence is guaranteed by Lemma \ref{lem_2326}. By the conclusions of Lemma \ref{lem_1816} and Lemma \ref{lem_2326},
for $t = t_0 \sim 1$,
\begin{equation}  \EE A_{E,t} \geq Proj_E \cdot \EE A_t \cdot Proj_E \geq \frac{1}{4} \cdot Proj_E,
\label{eq_1147} \end{equation}
where $t_0 = c_0$ is the universal constant from (\ref{eq_2310}). Set $m = n/3$  and $\nu = (Proj_E)_* \mu$. Select an orthonormal basis in $E$ and use it to identify $E \cong \RR^m$.
Then $\nu$ is an isotropic, log-concave probability measure in $\RR^m$ satisfying (\ref{eq_1148}), 
thanks to (\ref{eq_1146}) and (\ref{eq_1147}). Moreover, by Lemma \ref{lem_1148},
$$ L_{\nu} = L_{(Proj_E)_* \mu} \geq c L_n, $$
thus completing the proof.
\end{proof}

\section{Stability in the Shannon-Stam Inequality}
\label{sec_1547}

Throughout this section, we assume  that $\mu$ is an isotropic, log-concave probability measure in $\RR^n$. 
Recall from Section \ref{sec_1116} that for $t \geq 0$ and $\theta \in \RR^n$ we consider a certain probability density
$$ p_{t, \theta}: \RR^n \rightarrow [0, \infty) $$
whose barycenter and covariance are denoted by $a(t, \theta) \in \RR^n$ and $A(t, \theta) \in \RR^{n \times n}$.
We  switch to the F\"ollmer drift normalization used by Eldan and Mikulincer \cite{EM}, in order 
to apply their results on stability in the Shannon-Stam inequality.  We change variables as follows:
\begin{equation}  r = \frac{t}{t+1} \qquad \text{and} \qquad x = \frac{\theta}{1+t} = (1-r) \theta. 
\label{eq_1359} \end{equation}
With this change of variables, denote
\begin{equation}  v_r(x) = (1+t) a(t,\theta) - \theta \qquad \qquad \qquad (0 \leq r < 1, x \in \RR^n).  \label{eq_1449} \end{equation}
Let $(W_t)_{t \geq 0}$ be a standard Brownian motion in $\RR^n$ with $W_0 = 0$, 
and let $X$ be a random vector with law $\mu$ that is independent of the Brownian motion.
Recall the stochastic process $(\theta_t)_{t \geq 0}$ from Section \ref{sec_1116}, that coincides in law with the process $(t X +  W_t)_{t \geq 0}$.
For $r \in [0,1)$ we use the change of variables (\ref{eq_1359}) and set
\begin{equation}  X_r = (1 - r) \theta_{t} = (1 - r) \theta_{r/(1-r)}. \label{eq_1636} \end{equation}
Thus the process $(X_r)_{0 \leq r < 1}$ coincides in law with the process $( r X + (1-r) W_{r/(1-r)} )_{0 \leq r < 1}$. 
We remark that the law of the process $( (1-r) W_{r/(1-r)} )_{0 \leq r < 1}$ is that of the Brownian bridge.

\medskip 
Let $Z$ be a standard Gaussian random vector in $\RR^n$ that is independent of $X$. For any fixed $r \in [0,1)$, the 
random vector $X_r$ coincides in law with 
\begin{equation}  r X + \sqrt{r(1-r)} Z. \label{eq_1635} \end{equation}
Hence in some computations we may use the random vector in (\ref{eq_1635}) 
as a substitute for $X_r$.  The first and second moments of $X_r$ are given by 
\begin{equation} \EE X_r = 0 \qquad \text{and} \qquad\cov(X_r) = r^2 \cdot \id + r(1-r) \cdot \id = r \cdot \id. 
\label{eq_1641} \end{equation}
Abbreviate $$ v_r = v_r(X_r). $$
We compute that with the change of variables in (\ref{eq_1359}),
\begin{align}
v_r = v_r(X_r)  = 
v_r((1 - r) \theta_{t}) = (1+t) a(t, \theta_t) - \theta_{t} = (1+t) a_t - \theta_t. \label{eq_2252}
\end{align}

\medskip For two probability measures $\nu, \eta$ in $\RR^n$ for which $f = d \nu / d \eta$ exists and is differentiable  $\eta$-almost everywhere,
the Fisher information of  $\nu$ relative to $\eta$ is
$$ J(\nu \D \eta) = \int_{\RR^n} \frac{|\nabla f|^2}{f} d \eta = \int_{\RR^n} |\nabla \log f|^2 d \nu . $$
Recall that $\gamma_r$ is the Gaussian measure of mean zero and covariance $r \cdot \id$ in $\RR^n$.
Let us write $\nu_r$ for the probability measure in $\RR^n$ which is the law of $X_r$. 

\begin{lemma} For $r \in (0,1)$,
\begin{equation}  \EE |v_r|^2 = J(\nu_r \D \gamma_r).
\label{eq_1701} 
\end{equation}
Moreover, 
\begin{equation}  \EE |v_r|^2 \leq \frac{4n}{(1-r)^2}. \label{eq_1626} \end{equation}
	\label{lem_515}
\end{lemma}

\begin{proof} Let $t >0, \theta \in \RR^n$. Recall from (\ref{eq_1446}) that with $s = 1/t$ and $y = s \theta$,
\begin{equation}  a(t, \theta) = \frac{1}{t} \left( \theta - \nabla \psi_s(y) \right)
= \frac{1}{t} \left( \theta - \nabla \psi_{1/t}(\theta / t) \right), \label{eq_1454} \end{equation}
where $e^{-\psi_s}$ is the density of the random vector $X + \sqrt{s} Z$. Hence, by (\ref{eq_1359}), (\ref{eq_1449}) and (\ref{eq_1454}),
for $x \in \RR^n$ and $0 \leq r < 1$,
\begin{equation} v_r(x) = \frac{1}{r} \left( \frac{x}{1-r} - \nabla \psi_{\frac{1-r}{r}} \left( \frac{x}{r} \right) \right) - \frac{x}{1-r} = \frac{x}{r} - \frac{1}{r} \cdot \nabla \psi_{\frac{1-r}{r}} \left( \frac{x}{r} \right).  \label{eq_1500} \end{equation}
The function $\exp(-\psi_{(1-r)/r)})$ is proportional to the density of the random vector $X_r / r = X + \sqrt{(1-r)/r} Z$.
Hence, the density of the random vector $X_r$ is proportional to the function $$ x \mapsto \exp \left(-\psi_{\frac{1-r}{r}} \left( \frac{x}{r} \right) \right). $$
Thus, if $f = d \nu_r / d \gamma_r$ then we learn from (\ref{eq_1500}) that 
$$ v_r = \nabla \log f, $$
where we recall that $\nu_r$ is the law of $X_r$. Therefore,
$$ J(\nu_r \D \gamma_r) = \int_{\RR^n} |\nabla \log f|^2 d \nu_r = \int_{\RR^n} |v_r(x)|^2 d \nu_r(x) = \EE |v_r(X_r)|^2 = \EE |v_r|^2, $$
completing the proof of (\ref{eq_1701}). For the ``Moreover'' part, note that by (\ref{eq_2252}),
$$  \EE |v_r|^2 = \EE \left| (1+t) a_t - \theta_t \right|^2 \leq 2 (1+t)^2 \EE |a_t|^2 + 2 \EE |\theta_t|^2. $$
Since $\mu$ is isotropic, by (\ref{eq_2309}) we know that $\EE |a_t|^2 \leq n$. Additionally, $\EE |\theta_t|^2 = n (t + t^2)$. Hence,
$$  \EE |v_r|^2 \leq 2 (1+t)^2 n + 2 n (t + t^2) \leq 4 n (1 + t)^2 = \frac{4n}{(1-r)^2}. $$
\end{proof}

For two probability measures $\nu, \eta$ in $\RR^n$ for which $f = d \nu / d \eta$ exists, we write
\begin{equation}  D( \nu \D \eta ) = \int_{\RR^n} f \log f d \eta = \int_{\RR^n} (\log f) d \nu \in [0, +\infty] \label{eq_1614} \end{equation}
for the Kullback-Leibler divergence, also known as the relative entropy. 
The de Bruijn identity (see the variant in Klartag and Ordentlich  \cite[Proposition 1.5]{KO}) states that if $D(\nu \D \eta) < + \infty$ 
and $\int |x|^4 d \eta(x) < \infty$, then
for all $s > 0$,
\begin{equation}
\frac{d}{ds} D( \nu * \gamma_s \D \eta * \gamma_s ) = -\frac{1}{2} \cdot J( \nu * \gamma_s \D \eta * \gamma_s ),
\label{eq_1534} 
\end{equation}
and moreover, the expression on the right-hand side of (\ref{eq_1534}) is locally integrable in $s \in (0, \infty)$. 
In other words, the Kullback-Leibler divergence decays under Gaussian convolution, and the rate of decay is governed 
by the Fisher information. 
 Assuming  that $D(\nu * \gamma_s \D \eta * \gamma_s) \longrightarrow 0$ as $s \rightarrow \infty$
while $D(\nu * \gamma_s \D \eta * \gamma_s) \longrightarrow D(\nu \D \eta)$ as $s \rightarrow 0$,  we have
\begin{equation}  D(\nu \D \eta)  =\frac{1}{2} \int_0^{\infty} J( \nu * \gamma_s \D \eta * \gamma_s ) ds. \label{eq_1546} \end{equation}
When $U$ and $V$ are random vectors with laws $\nu$ and $\eta$ respectively, we write $D(U \D V) = D(\nu \D \eta)$
and $J(U \D V) = J(\nu \D \eta)$. Recall that $\gamma_1$ is the standard Gaussian probability measure in $\RR^n$.

\begin{corollary} With the above notation,
\begin{equation}  D(\mu \D \gamma_1) = \frac{1}{2}  \int_0^1 J(\nu_r \D \gamma_r) dr
= \frac{1}{2}  \int_0^1 \EE |v_r|^2  dr. \label{eq_1754} 
\end{equation}
Additionally, since $\mu$ is isotropic,
\begin{equation} D(\mu \D \gamma_1) = -\Ent(\mu) + \frac{n}{2} \log(2 \pi e). \label{eq_1642} \end{equation}
\label{cor_1757}
\end{corollary}

\begin{proof} Since $X$ has law $\mu$ and $Z$ has law $\gamma_1$, 
$$ D(X \D Z) = D(\mu \D \gamma_1) \qquad \text{and} \qquad J(X \D Z) = J(\mu \D \gamma_1). $$
Note that $D(\alpha X \D \alpha Z) = D(X \D Z)$ while $J(\alpha X \D \alpha Z) = \alpha^{-2} \cdot J(X \D Z)$ for any $\alpha > 0$.
With $s = (1-r) / r$ we observe that 
\begin{align}  \nonumber J(X + \sqrt{s} Z \D \sqrt{s+1} Z) 
& = J(X + \sqrt{(1-r)/r} Z \D Z/\sqrt{r})  \\ & 
= r^2 \cdot J(r X + \sqrt{r(1-r)} Z \D \sqrt{r} Z)  = r^2 \cdot J(\nu_r \D \gamma_r). 
\label{eq_1627} \end{align} 
From the de Bruijn identity (\ref{eq_1546}),
\begin{equation}  D(\mu \D \gamma_1) = \frac{1}{2} \int_0^{\infty} J( \mu * \gamma_s \D \gamma_1 * \gamma_s) ds = \frac{1}{2} \int_0^{\infty} J( X + \sqrt{s} Z \D \sqrt{s+1} Z) ds. \label{eq_1547} \end{equation}
Indeed, equation (\ref{eq_1547}) holds true since it is straightforward to verify, using the dominated convergence theorem, 
that $D(Z + \eps X \D Z + \eps Z')$
tends to zero as $\eps \rightarrow 0$ while $D(X + \eps Z' \D Z + \eps Z')$ tends to $D(X \D Z)$ as $\eps \rightarrow 0$.
Here $Z'$ is a standard Gaussian random vector in $\RR^n$, independent of $X$ and $Z$.
Changing variables $s = (1-r) / r$ in the integral in (\ref{eq_1547}) and using (\ref{eq_1627}) we obtain
$$ D(\mu \D \gamma_1) = \frac{1}{2}  \int_0^{\infty} J( X + \sqrt{s} Z \D \sqrt{s+1} Z) \frac{dr}{r^2} = 
\frac{1}{2} \int_0^1 J(\nu_r \D \gamma_r) dr.$$
This proves the first equality in (\ref{eq_1754}), with the second equality being the content of 
Lemma~\ref{lem_515}. In order to prove the ``Moreover'' part, we note that, with $$ \gamma(x) = (2 \pi)^{-n/2} \exp(-|x|^2/2) $$ being the standard Gaussian density 
in $\RR^n$, and with $\rho$ being the density of $\mu$ with respect to the Lebesgue measure in $\RR^n$,
\begin{align*}
D(\mu \D \gamma_1) & = \int_{\RR^n} \log \frac{\rho}{\gamma} d \mu = \int_{\RR^n} (\log \rho) d \mu + \frac{n}{2} \log(2 \pi) + \int_{\RR^n} \frac{|x|^2}{2} d \mu(x)
\\ & = -\Ent(\mu) + \frac{n}{2} \log(2 \pi) + \frac{n}{2}. %\tag*{\qedhere}
\end{align*}
\end{proof}

 We use the change of variables  (\ref{eq_1359}) once more and set 
\begin{equation}  \Gamma_r(x) = (1 + t)  A(t, \theta) \geq 0 \qquad \qquad \qquad (0 \leq r < 1, x \in \RR^n). \label{eq_1623} \end{equation}
Abbreviate $ \Gamma_r = \Gamma_r(X_r) $ so that by  (\ref{eq_1636}) and (\ref{eq_1623}),
\begin{equation}  \Gamma_r =  \Gamma_r(X_r) = (1+t) A(t, X_r/(1-r) ) = (1+t) A(t, \theta_t) = (1+t) A_t, \label{eq_2018}
\end{equation}
where $(A_t)_{t \geq 0}$ is the covariance process associated with  stochastic localization starting from~$\mu$. 
The following lemma summarizes the basic properties of the matrix process $(\Gamma_r)_{0 \leq r < 1}$.

\begin{lemma} Let $t \geq 0$ and  $r = t / (t+1)$. Then the following hold:
\begin{enumerate}
\item[(i)] $\displaystyle (1-r)  \EE \Gamma_r = \EE A_t$.
\item[(ii)] $\displaystyle \EE v_r \otimes v_r = \frac{\id - \EE \Gamma_r}{1-r}$, and consequently $0 \leq \EE \Gamma_r \leq \id$.
\item[(iii)] $\displaystyle  \frac{d}{dr} \EE v_r \otimes v_r =  \frac{\EE(\id - \Gamma_r)^2}{(1-r)^2} $.
\item[(iv)] $\displaystyle \frac{d}{dr} \EE \Gamma_r = \frac{\EE \Gamma_r - \EE \Gamma_r^2}{1-r}$.
\item[(v)] Almost surely, $\displaystyle \Gamma_r \leq \frac{1}{r} \cdot \id$.
\end{enumerate}
\label{lem_1709}
\end{lemma}

\begin{proof} We obtain (i) by taking the expectation of (\ref{eq_2018}). Next, from (i) and (\ref{eq_2309}),
\begin{equation}  (1-r) \EE \Gamma_r + \EE a_t \otimes a_t = \id. \label{eq_1643} \end{equation}
From  (\ref{eq_1636}) and (\ref{eq_2252}), and by using Lemma \ref{lem_1213},
\begin{equation}  \EE v_r \otimes X_r = (1+t) (1-r) \cdot \EE \left[ a_t - \frac{\theta_t}{t+1} \right] \otimes \theta_t = 0. 
\label{eq_1516} \end{equation}
Recall from (\ref{eq_1641}) that $\EE X_r \otimes X_r = r \cdot \id$
and from (\ref{eq_1636})  that  $a_t = (1-r) v_r + X_r$. It thus follows from (\ref{eq_1516}) 
that 
\begin{align} \nonumber
\EE a_t \otimes a_t & =  \EE \left[ ((1-r) v_r + X_r) \otimes ((1-r) v_r + X_r) \right] 
 \\ & = (1-r)^2 \EE v_r \otimes v_r + \EE X_r \otimes X_r = (1-r)^2 \EE v_r \otimes v_r + r \cdot \id.
 \label{eq_1639}
 \end{align}
Now (ii) follows from (\ref{eq_1643}) and (\ref{eq_1639}). Next, by (i),
\begin{equation}  \frac{d}{dr} \EE \Gamma_r = \frac{d}{dr}  \frac{\EE A_t}{1-r} = 
\frac{\EE A_t}{(1-r)^2} + \frac{1}{1-r} \frac{d}{dr} \EE A_t = \frac{\EE \Gamma_r}{1-r} + \frac{1}{(1-r)^3} \frac{d}{dt} \EE A_t^2, \label{eq_1624} \end{equation}
where we used $dr / dt = (1-r)^2$ in the last passage. Thus, by (\ref{eq_1645}) and (\ref{eq_1624}),
$$    \frac{d}{dr} \EE \Gamma_r   = \frac{\EE \Gamma_r}{1-r} - \frac{1}{(1-r)^3} \EE A_t^2 =
\frac{\EE \Gamma_r}{1-r} - \frac{1}{1-r}  \EE \Gamma_r^2, $$
proving (iv). Item (iii) follows from (ii) and (iv) by a straightforward computation.
In order to obtain (v), we use (\ref{eq_1711}) and (\ref{eq_2018}) which give
$$ \Gamma_r = (1 + t)  A_t  \leq \frac{1+t}{t} \cdot \id = \frac{1}{r} \cdot \id. $$
\end{proof}

For two probability measures $\mu^{(1)}$ and $\mu^{(2)}$ in $\RR^n$ for which $D(\mu^{(i)} \D \gamma_1) < \infty$ for $i=1,2$ we write
\begin{equation}  \deltaEPI( \mu^{(1)}, \mu^{(2)} ) = \frac{D( X^{(1)} \D Z) + D( X^{(2)} \D Z)}{2} - D \left( \frac{X^{(1)} + X^{(2)}}{\sqrt{2}} \D Z \right) 
\label{eq_2100} \end{equation}
where $X^{(1)}$ and $X^{(2)}$ are two independent random vectors in $\RR^n$, with laws $\mu^{(1)}$ and $\mu^{(2)}$ respectively, 
and where $Z$ is a standard Gaussian random vector in $\RR^n$. The Shannon-Stam inequality  implies that $\deltaEPI( \mu^{(1)}, \mu^{(2)} )$ is non-negative,
with quantitative stability estimates in terms of relative entropy by Eldan and Mikulincer \cite{EM}. 
We consider the case where
$ \mu^{(1)} = \mu^{(2)}$ and abbreviate
$$ \deltaEPI( \mu ) = \deltaEPI( \mu, \mu ). $$
 In this particular case, Lemma 2 from  \cite{EM}
with $\lambda = 1/2$ yields the following:

\begin{lemma}[Eldan and Mikulincer] Let $(\Gamma_r^{(1)})_{0 \leq r < 1}$ and $(\Gamma_r^{(2)})_{0 \leq r < 1}$ be two independent copies 
of the matrix process $(\Gamma_r)_{0 \leq r < 1}$. Then,
$$ \deltaEPI( \mu) \geq \int_0^1 \frac{\EE  \Tr \left[ \left(\Gamma_r^{(1)} - \Gamma_r^{(2)} \right)^2 
\left( \sqrt{ ( (\Gamma_r^{(1)})^2 + (\Gamma_r^{(2)})^2 ) / 2 } + (\Gamma_r^{(1)} + \Gamma_r^{(2)} ) / 2  \right)^{-1}
\right] }{4 (1-r) } dr. $$
Consequently, suppose that $\xi \in (0,1)$ and $\eps > 0$ are fixed numbers such that $\Gamma_r \leq \eps^{-1} \cdot \id$ almost surely for all $r \in (\xi, 1)$. Then,
$$ \deltaEPI( \mu ) \geq \eps \cdot \int_\xi^1 \frac{\EE   \left|\Gamma_r^{(1)} - \Gamma_r^{(2)} \right|^2}{8 (1-r) } dr
= \eps \cdot \int_\xi^1 \frac{\EE   \left|\Gamma_r - \EE \Gamma_r \right|^2}{4 (1-r) } dr. $$
\label{lem_EM}
\end{lemma}

Ball and Nguyen \cite{BN} showed that the quantity 
$\deltaEPI(\mu)$ exhibits a particularly simple behavior in the log-concave case.

\begin{lemma}[Ball and Nguyen] Assume that the log-concave probability measure $\mu$ in $\RR^n$ is isotropic. Then,
$$ \deltaEPI( \mu ) \leq 2 n. $$
\label{lem_BN}
\end{lemma}

\begin{proof} Write $f$ for the log-concave density of $\mu$. Let $X^{(1)}$ and $X^{(2)}$ be two independent random vectors with law $\mu$. By (\ref{eq_2100}) and 
Corollary \ref{cor_1757},
\begin{align*}  \deltaEPI( \mu ) & = \Ent \left( \frac{X^{(1)} + X^{(2)}}{\sqrt{2}} \right)  - \Ent ( X^{(1)} ).
\end{align*}
As is proved in the third displayed formula after (5.3) in  Ball and Nguyen \cite{BN},
$$ \Ent \left( \frac{X^{(1)} + X^{(2)}}{\sqrt{2}} \right) \leq -\log f(0) + 2n. $$
However, $\Ent ( X^{(1)} ) = \Ent(\mu) \geq -\log f(0)$ by (\ref{eq_1048}), and the lemma follows.
\end{proof}

\section{Proof of Theorem \ref{thm_1744}}
\label{sec_thmproof}

We adapt the argument from Eldan and Mikulincer \cite{EM}.
Assume that $ n \geq 10$ 
and let us apply Proposition \ref{prop_1115}. From the conclusion of the proposition, there exists 
an integer $m \in [n/4,n]$
and an isotropic, log-concave probability measure $\mu$ in $\RR^m$ 
with 
\begin{equation}
L_{\mu} \gtrsim L_n \label{eq_1722}
\end{equation}
and such that for some $t_0 \sim 1$, 
\begin{equation}  \EE A_{t_0} \geq \frac{1}{4} \cdot \id.  \label{eq_1723} \end{equation}
Here, $(A_t)_{t \geq 0}$ is the covariance process associated with stochastic localization starting from the measure $\mu$.
Denote 
$$ \xi = \frac{t_0}{t_0 + 1} $$
so that 
\begin{equation}  c_1 \leq \xi \leq 1 - c_1 \label{eq_1727} \end{equation}
for some universal constant $c_1 > 0$. From (\ref{eq_1723}) and Lemma \ref{lem_1709}(i),
\begin{equation}
 \EE \Gamma_\xi = \frac{1}{1 - \xi} \EE A_{t_0} \geq \frac{1}{4} \cdot \id. 
 \label{eq_1724}
 \end{equation}
Applying Lemma \ref{lem_1709}(iv,v) we see that
$$ \frac{d}{dr} \EE \Gamma_r = \frac{\EE \Gamma_r - \EE \Gamma_r^2}{1-r} \geq \frac{\EE \Gamma_r - \EE \Gamma_r / r}{1-r} = -\frac{\EE \Gamma_r}{r}. $$
Therefore
$$ \frac{d}{dr} (r \cdot \EE \Gamma_r) \geq 0, $$
and $r \cdot \EE \Gamma_r$ is increasing in $r \in [0, 1)$. Consequently, from (\ref{eq_1727}), (\ref{eq_1724}) and Lemma \ref{lem_1709}(ii), for any $r \geq \xi$,
\begin{equation}
\tilde{c} \cdot \id \leq \EE \Gamma_r \leq \id.
\label{eq_1728}
\end{equation}
By Lemma \ref{lem_1709}(ii) and (\ref{eq_1728}), for $r \geq \xi$,
\begin{align} \nonumber 
\frac{|\id - \EE \Gamma_r|^2}{1-r} & = \Tr \left[ \frac{\id - \EE \Gamma_r}{1-r} \cdot (\id - \EE \Gamma_r) \right]
 = \Tr \left[ (\EE v_r \otimes v_r) \cdot (\id - \EE \Gamma_r) \right] \\ & \leq (1 - \tilde{c}) \Tr \left[ \EE v_r \otimes v_r \right] = (1 - \tilde{c}) \EE |v_r|^2,
\label{eq_1725}
\end{align}
where we used the fact that $\tr[AB] \geq 0$ for two symmetric, positive semi-definite matrices $A, B \in \RR^{m \times m}$.
By integration by parts and Lemma \ref{lem_1709}(iii),
\begin{align}   \nonumber \int_{\xi}^1 \EE |v_r|^2 dr & = \int_{\xi}^1 (1-r) \frac{d}{dr} \EE |v_r|^2 dr + (1 - \xi) \EE |v_\xi|^2
\\ & =   \int_{\xi}^1 \frac{\EE | \id - \Gamma_r |^2}{1-r} dr + (1 - \xi) \EE |v_{\xi}|^2.
\label{eq_1050}
\end{align}
By Lemma \ref{lem_1709}(v), almost surely $\Gamma_r \leq r^{-1} \cdot \id \leq \xi^{-1} \cdot \id$ for $r \in (\xi,1)$.
From Lemma \ref{lem_EM} and Lemma \ref{lem_BN} we thus see that 
\begin{equation}  2m \geq \deltaEPI( \mu ) \geq \xi \cdot \int_\xi^1 \frac{\EE   \left|\Gamma_r - \EE \Gamma_r \right|^2}{4 (1-r) } dr. \label{eq_2135} 
\end{equation}
For $r \geq \xi$ we decompose 
\begin{align}   \EE \left|\Gamma_r - \EE \Gamma_r \right|^2  = \EE \left|\id - \Gamma_r \right|^2 - \left|\id - \EE \Gamma_r \right|^2.
\label{eq_2136} \end{align}
Thus, by
(\ref{eq_1725}),  (\ref{eq_1050}), (\ref{eq_2135}) and (\ref{eq_2136}),
\begin{align} \nonumber \frac{8 m}{\xi} & \geq \int_\xi^1 \frac{\EE \left|\Gamma_r - \EE \Gamma_r \right|^2}{1-r} dr =  \int_\xi^1 \frac{\EE \left|\id - \Gamma_r\right|^2}{1-r} dr  - \int_{\xi}^1 \frac{\left|\id - \EE \Gamma_r \right|^2}{1-r} dr \\ \nonumber & \geq \int_\xi^1 \frac{\EE \left|\id - \Gamma_r\right|^2}{1-r} dr - (1 - \tilde{c}) \int_\xi^1 \EE |v_r|^2 dr \\ & =  \int_{\xi}^1 \EE |v_r|^2 dr -  (1 - \xi) \EE |v_{\xi}|^2 -  (1 - \tilde{c}) \int_\xi^1 \EE |v_r|^2 dr.
\label{eq_2147} 
\end{align}
By (\ref{eq_1727}) and (\ref{eq_2147}),
\begin{equation}
\int_{\xi}^1 \EE |v_r|^2 dr \leq \frac{1 - \xi}{\tilde{c}} \cdot \EE |v_{\xi}|^2 + C m. 
\label{eq_2148}
\end{equation}
Since $1 - \xi \geq c_1$, the ``Moreover'' part of Lemma \ref{lem_515} tells us that $\EE |v_{\xi}|^2 \leq \tilde{C} m$. Hence,
from (\ref{eq_2148}),
\begin{equation}
\int_{\xi}^1 \EE |v_r|^2 dr \leq \bar{C} m.
\label{eq_2148_}
\end{equation}
By Lemma \ref{lem_1709}(iii) we know that $r \mapsto \EE |v_r|^2$ is non-decreasing. Hence,
from (\ref{eq_1727}) and (\ref{eq_2148_}),
\begin{equation}  \int_{0}^1 \EE |v_r|^2 dr \leq \frac{1}{1-\xi} \int_{\xi}^1 \EE |v_r|^2 dr \leq C m.
\label{eq_1758} \end{equation}
Thus, from (\ref{eq_1758}) 
and Corollary \ref{cor_1757},
$$ -\Ent(\mu) + \frac{m}{2} \log(2 \pi e) = D(\mu \D \gamma_1) 
= \frac{1}{2} \int_{0}^1 \EE |v_r|^2 dr \leq C m. $$
Consequently,
\begin{equation}  \Ent(\mu) \geq - \tilde{C} m. \label{eq_1707} \end{equation}
Since $\cov(\mu) = \id$, from (\ref{eq_1750}) and (\ref{eq_1707}) we  deduce that 
$$ L_{\mu} \leq \bar{C}. $$
The conclusion of Theorem \ref{thm_1744} thus follows from (\ref{eq_1722}). \hfill \qed

\medskip
\noindent Department of Mathematics,
Weizmann Institute of Science,
Rehovot 76100, Israel. \\
{\it e-mail:} \verb"boaz.klartag@weizmann.ac.il"

\medskip
\noindent Universit\'e de Poitiers, CNRS, LMA, Poitiers, France. \\
{\it e-mail:} \verb"joseph.lehec@univ-poitiers.fr"

%\vfill \hfill \today
\end{document}